\documentclass[smallextended]{svjour3}
\usepackage{amsmath,amssymb,mathabx,amsrefs,mathrsfs,lineno,fancyvrb,tikz,wrapfig,bbold,epigraph,mathtools}
\usepackage{hyperref}
\usetikzlibrary{arrows,automata,shapes}

\font\cyreight=wncyr8

\newcommand\N{{\mathbb N}}
\newcommand\Z{{\mathbb Z}}
\newcommand\R{{\mathbb R}}
\newcommand\one{{\mathbb 1}}

\DeclareMathOperator\sym{Sym}

\newtheorem{algorithm}{Algorithm}
\newtheorem{assumption}{Standing assumption}

\newenvironment{fsa}[1][auto]{\begin{tikzpicture}[->,>=stealth',
    shorten >=1pt,auto,node distance=3cm,double distance between line centers=0.45ex,
    initial text=,accepting/.style=accepting by arrow,
    every state/.style={inner sep=3pt,minimum size=0pt},
    every loop/.style={looseness=12},semithick,#1]}{\end{tikzpicture}}

\newcommand\autsq[5]{
  \begin{tikzpicture}[->,>=stealth',auto,semithick,baseline=2mm]
      \draw[->] (1,0) -- node[right] {$#4$} (1,1);
      \draw[->] (0,0) -- node[below] {\smash{\raisebox{-1.2ex}{$#2$}}} (1,0);
      \draw[->] (0,0) -- node{$#3$} (0,1);
      \draw[->] (0,1) -- node[above] {\smash{\raisebox{-0.4ex}{$#5$}}} (1,1);
      \node at (0.5,0.5) {$#1$};
    \end{tikzpicture}}

\begin{document}
\title{Decidability problems in automaton semigroups\thanks{Partially supported by ANR grant ANR-14-ACHN-0018-01}}
\author{Laurent Bartholdi}
\institute{\'Ecole Normale Sup\'erieure, Paris. \email{laurent.bartholdi@ens.fr}\and Georg-August-Universit\"at zu G\"ottingen. \email{laurent.bartholdi@gmail.com}}
\date{4 December 2016}
\maketitle
\begin{abstract}
  We consider decidability problems in self-similar semigroups, and in
  particular in semigroups of automatic transformations of $X^*$.

  We describe algorithms answering the word problem, and bound its
  complexity under some additional assumptions.

  We give a partial algorithm that decides in a group generated by an
  automaton, given $x,y$, whether an Engel identity
  ($[\cdots[[x,y],y],\dots,y]=1$ for a long enough commutator
  sequence) is satisfied. This algorithm succeeds, importantly, in
  proving that Grigorchuk's $2$-group is not Engel.

  We consider next the problem of recognizing Engel elements, namely
  elements $y$ such that the map $x\mapsto[x,y]$ attracts to $\{1\}$.
  Although this problem seems intractable in general, we prove that it
  is decidable for Grigorchuk's group: Engel elements are precisely
  those of order at most $2$.

  We include, in the text, a large number of open problems.  Our
  computations were implemented using the package \textsc{Fr} within
  the computer algebra system \textsc{Gap}.
\end{abstract}

%%%%%%%%%%%%%%%%%%%%%%%%%%%%%%%%%%%%%%%%%%%%%%%%%%%%%%%%%%%%%%%%
\section{Introduction}
\epigraph{Automata are infinity --- come down to computer scientists' level}{Louis-Ferdinand C\'eline, \emph{Journey to the End of the Night}}

The theory of groups, and even more so of semigroups, is fundamentally
example-driven: on the one hand, these algebraic objects encode the
symmetry, regularity, and operations present in any kind of structure
under consideration, and are thus given to us for study; on the other
hand, there is such a diversity of groups and semigroups that the most
one can hope for, in their theory, is a description of the phenomena that
may occur.

Groups and semigroups are fundamentally (semi)groups of self-maps of a
set.  To construct infinite (semi)groups, we should therefore give
oneself an infinite set and a collection of self-maps, or possibly a
generating set of self-maps. Automata are of fundamental use, in the
guise of \emph{transducers}, in giving finite, recursive descriptions
of infinite self-maps; see~\cite{gecseg-c:ata}.

Now, even if the description of the (semi)group's generators are
completely explicit, this does not mean that the (semi)group is well
understood. For example, one may want to know, given two words $u,v$
representing products of generators, whether they are \emph{equal} in
the semigroup; \emph{semiconjugate} ($\exists w:u w\equiv w v$),
\emph{conjugate} ($\exists\text{ invertible }w:u w\equiv w v$), etc. These
\emph{decision problems} are, usually, undecidable, and the question
is which extra conditions on the semigroup's generators guarantee the
problem's decidability.

In this text, I will survey a general construction of
\emph{self-similar} semigroups, highlight important decision problems,
and describe stronger and stronger restrictions on the self-similarity
structure in parallel with solutions to decision problems. There is a
wealth of unsolved problems in this area, and I hope that the panorama
provided by this text will have some value in highlighting interesting,
yet-unexplored areas of mathematics and theoretical computer science.

The new results included in this text are in particular a partial
algorithm that answers, in self-similar groups, whether the Engel
property holds (for all $x,y\in G$, some long-enough iterated
commutator $[\cdots[x,y],\dots,y]$ is trivial). Remarkably, this
partial algorithm, once implemented, proved that the first Grigorchuk
group is not Engel.

%%%%%%%%%%%%%%%%%%%%%%%%%%%%%%%%%%%%%%%%%%%%%%%%%%%%%%%%%%%%%%%%
\section{Self-similar semigroups}
A \emph{self-similar semigroup} is a semigroup acting in a
self-similar fashion on a self-similar set.

We are thus given a set $\Omega$ and a finite family of self-maps
$X=\{x\colon\Omega\righttoleftarrow\}$. The set $\Omega$ is
self-similar in the sense that subsets $x(\Omega)\subseteq\Omega$ are
defined and identified (via $x$) with $\Omega$, for a collection of
$x\in X$. Such systems are often called \emph{iterated function
  systems}.  The fundamental example is $\Omega=X^\infty$ the set of
infinite words over an alphabet $X$, and in that case $x\in X$ acts on
an infinite word by pre-catenation: $x(x_1x_2\dots)=x x_1x_2\dots$;
another example is $\Omega=X^*$, the tree of finite words, with again
the same identification of $X$ with self-maps of $\Omega$ by
pre-catenation.

\begin{definition}[\cite{nekrashevych:ssg}]\label{def:ss}
  Let $G$ be a semigroup acting on the right on a set $\Omega$. The
  action is called \emph{self-similar} if for every $x\in X,g\in G$
  there exist $h\in G,y\in X$ such that
  \begin{equation}\label{eq:ss}
    x(\omega)^g=y(\omega^h)\text{ for all }\omega\in\Omega.
  \end{equation}
\end{definition}

In other words, the action of $g\in G$ on $\Omega$ is as follows: it
carries the subset $x(\Omega)$ to $y(\Omega)$, and along the way
transforms $\Omega$ by $h$.
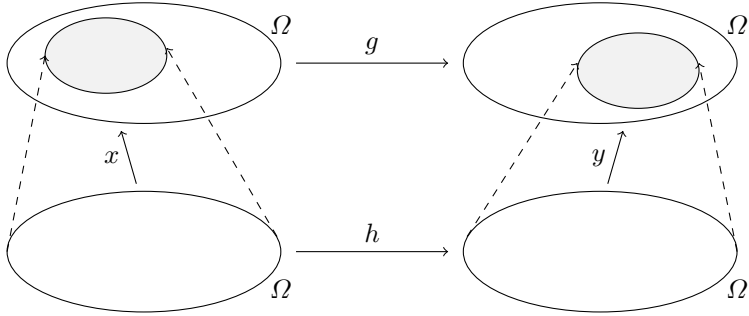
\[\begin{tikzpicture}
    \draw (-3,0) ellipse (18mm and 8mm);
    \draw (3,0) ellipse (18mm and 8mm);
    \draw (-3,2.5) ellipse (18mm and 8mm);
    \draw (3,2.5) ellipse (18mm and 8mm);
    \draw[fill=gray!10] (-3.5,2.6) ellipse (8mm and 5mm);
    \draw[fill=gray!10] (3.5,2.4) ellipse (8mm and 5mm);
    \node at (-1.2,3) {$\Omega$};
    \node at (-1.2,-0.5) {$\Omega$};
    \node at (4.8,3) {$\Omega$};
    \node at (4.8,-0.5) {$\Omega$};
    \draw (-1,0) edge[->] node[above] {$h$} +(2,0);
    \draw (-1,2.5) edge[->] node[above] {$g$} +(2,0);
    \draw (-4.8,0) edge[dashed,->] (-4.3,2.6);
    \draw (-1.25,0.2) edge[dashed,->] (-2.7,2.7);
    \draw (1.25,0.2) edge[dashed,->] (2.7,2.5);
    \draw (4.8,0) edge[dashed,->] (4.3,2.5);
    \draw (-3.1,0.9) edge[->] node[left] {$x$} (-3.3,1.6);
    \draw (3.1,0.9) edge[->] node[left] {$y$} (3.3,1.6);
  \end{tikzpicture}
\]

If we were to write the function application $x(\omega)$ on the right,
as $(\omega)x$, then we could rephrase~\eqref{eq:ss} as ``$x g=h y$''
qua composition of self-maps of $\Omega$.

The elements $h\in G,y\in X$ are not necessarily unique in
Definition~\ref{def:ss}. Let us assume that some choices are made for
them; then they may be encoded in a map
$\Phi\colon X\times G\to G\times X$, given by $(x,g)\mapsto (h,y)$
when~\eqref{eq:ss} holds. This map satisfies some axioms following
from the fact that $G$ is a semigroup acting on $\Omega$. We summarize
them in the following
\begin{definition}
  A \emph{self-similarity structure} for a semigroup $G$ is the data
  of a set $X$ and a map $\Phi\colon X\times G\to G\times X$ satisfying
  \[\Phi(x,\one)=(\one,x),\quad\Phi(x,g)=(h,y)\wedge\Phi(y,g')=(h',z)\Rightarrow\Phi(x,g g')=(h h',z).\]
\end{definition}

From a self-similarity structure, one can reconstruct a self-similar
action on $\Omega=X^*$ by defining recursively (for
$\varepsilon$ the empty word in $X^*$)
\begin{equation}\label{eq:action}
  \varepsilon^g=\varepsilon,\qquad(x w)^g=y(w^h)\text{ whenever }\Phi(x,g)=(h,y).
\end{equation}
The action on $X^*$ extends uniquely by continuity to an action on
infinite words $X^\infty$.

A self-similar semigroup may be \emph{defined} by specifying a
generating set $Q$, an alphabet $X$, and a map
$\Phi\colon X\times Q\to Q^*\times X$. The semigroup defined is then
the semigroup of self-maps of $X^*$ given by~\eqref{eq:action}. We
write that semigroup
\[G(\Phi).\]

It is quite convenient to describe a self-similar semigroup
$G=\langle Q\rangle$ by writing its self-similarity structure on the
perimeters of squares: there is a square for each $x\in X,g\in Q$; the
left, bottom, right, top labels are respectively $x,g,y,h$ when
$\Phi(x,g)=(h,y)$. In this manner, to compute the action of
$g=q_1\dots q_n$ on a word $x_1\dots x_m$, one writes $x_1\dots x_m$
on the left and $q_1\dots q_n$ on the bottom of an $m\times n$
rectangle, and one fills in the rectangle's squares one at a time. The
right label will then be $y_1\dots y_m$, the image of $x_1\dots x_m$
under $g$. Here are two examples of self-similar semigroups given in
this manner:

\begin{example}[Grigorchuk's group]\label{ex:grig}
  The generating set is $Q=\{a,b,c,d,\one\}$ and the alphabet is
  $X=\{1,2\}$. The self-similarity structure is given by
  \begin{xalignat*}{5}
    \autsq\Phi a12\one && \autsq\Phi b11a && \autsq\Phi c11a && \autsq\Phi
    d11\one && \autsq\Phi \one11\one,\\[1ex]
    \autsq\Phi a21\one && \autsq\Phi b22c && \autsq\Phi c22d && \autsq\Phi
    d22b && \autsq\Phi \one22\one.
  \end{xalignat*}
  The associated automaton is depicted in Figure~\ref{fig:automata},
  left.  The Grigorchuk group will be denoted by $G_0$ throughout this
  text.

  It is a remarkable example of group: among its properties, it is an
  \emph{infinite, finitely generated torsion group}, namely, the group
  is infinite, but every element generates a finite subgroup. It is
  also a group of \emph{intermediate word-growth}, namely, the number
  $v(n)$ of group elements that are products of at most $n$ generators
  is a function growing asymptotically as
  \[\exp(n^{0.51})\prec v(n)\prec\exp(n^{0.76}).\]
  (the exact growth asymptotics are not known; see~\cites{bartholdi:lowerbd,brieussel:phd,bartholdi:upperbd}.)
\end{example}

\begin{example}[A two-state automaton of intermediate growth]\label{ex:2}
  The generating set is $\{q,r\}$, and the alphabet is
  $X=\{1,2\}$. The self-similarity structure is given by
  \begin{xalignat*}{4}
    \autsq\Phi r12r && \autsq\Phi r22s && \autsq\Phi s12s && \autsq\Phi s21s.
  \end{xalignat*}
  The associated automaton is depicted in Figure~\ref{fig:automata},
  right.  The semigroup $G_1=\langle r,s\rangle$ is infinite, and the
  growth of $G_1$ is better understood than that of $G_0$: letting
  $v(n)$ denote the number of elements of $G_1$ that are products of
  at most $n$ generators $r,s$, we have
  \[v(n)\approx 2^{5/2}3^{3/4}\pi^{-2}n^{1/4}\exp(\pi\sqrt{n/6}).\]
\end{example}

Automata can be naturally \emph{composed}, in two manners;
see~\cite{gecseg:products}. Let us consider a single automaton $\Phi$,
with stateset $Q$ and alphabet $X$. Then, for every $m,n\in\N$, there
is an automaton $\Phi_{m,n}$ with stateset $Q^n$ and alphabet $X^m$,
described by squares as follows. For all words $g\in Q^n$ and
$u\in X^m$, one writes $g,u$ respectively at the left and bottom of an
$m\times n$ rectangle, and fills it by the $1\times 1$ squares of the
automaton $\Phi$.

The meaning of these automata is the following. In $\Phi_{m,1}$, the
automaton has stateset $Q$ and alphabet $X^m$: its arrows are
length-$m$ directed paths in the automaton $\Phi$. In formul\ae, this
automaton $\Phi_{m,1}$ is defined by
\[\Phi_{m,1}(x_1\dots x_m,s_0)=(s_m,y_1\dots y_m)\text{ if }\Phi(x_i,s_{i-1})=(s_i,y_i)\text{ for all }i=1,\dots,m.
\]
The automaton $\Phi_{m,1}$ expresses the action of $Q$ on words of
length (a multiple of) $m$. Similarly, the automaton $\Phi_{1,n}$ has
stateset $Q^n$ and alphabet $X$, and is given by
\[\Phi_{1,n}(x_0,s_1\dots s_n)=(t_1\dots t_n,x_n)\text{ if }\Phi(x_{i-1},s_i)=(t_i,x_i)\text{ for all }i=1,\dots,n.
\]
It expresses the action on $X^*$ of words of length $n$ in $Q$. These
products may naturally be combined so as to give an automaton
$\Phi_{m,n}$ with stateset $Q^n$ and alphabet $X^m$.

We shall abuse notation and write $\Phi(u,g)$ instead of
$\Phi_{m,n}(u,g)$, since it is always clear from the arguments $u,g$
what the values of $m,n$ are.

%%%%%%%%%%%%%%%%%%%%%%%%%%%%%%%%%%%%%%%%%%%%%%%%%%%%%%%%%%%%%%%%
\section{Decision problems}
The study of decision problems is commonly attributed to
Dehn~\cite{dehn:unendliche}, though its origins can be traced to
Hilbert's work. Let $G$ be a finitely generated semigroup, and
consider a finite generating set $Q$. There is therefore an
\emph{evaluation} map $Q^*\twoheadrightarrow G$, written
$w\mapsto\overline w$. Consider the following questions:
\renewcommand\descriptionlabel[1]{\textbf{#1}:}
\begin{description}
\item[Word problem (WP)] Given $u,v\in Q^*$, does one have
  $\overline u=\overline v$?
\item[Division problem] Given $u,v\in Q^*$, is $\overline u$ a left divisor of
  $\overline v$? I.e.\ does one have $\overline u G\ni\overline v$? Is
  it a right divisor?
\item[Order problem (OP)] Given $u\in Q^*$, is $\langle u\rangle$ finite? If so, what is
  its structure, i.e.\ what are the minimal $m<n$ with
  $\overline u^m=\overline u^n$?
\item[Inverse problem] Given $u\in Q^*$, is $\overline u$ invertible?
\item[Conjugacy problem] Given $u,v\in Q^*$, are they semiconjugate, i.e.\ is
  there $g\in G$ with $g\overline u=\overline v g$? Are they
  conjugate, i.e.\ is there an invertible $g\in G$ with
  $g\overline u=\overline v g$
\item[Membership problem (MP)] Given $u,v_1,\dots,v_n\in Q^*$, does one have
  $\overline u\in\langle\overline{v_1},\dots,\overline{v_n}\rangle$?
\item[Structure problem] Given $u_1,\dots,u_n\in Q^*$, is the semigroup
  $\langle\overline{u_1},\dots,\overline{u_n}\rangle$ free? Is it
  finite?
\item[Engel problem] Given $u,v\in Q^*$, and assuming
  $\overline u,\overline v$ are invertible, are they an \emph{Engel
    pair}, i.e.\ does there exist $n\in\N$ such that the $n$-fold
  iterated commutator satisfies
  $[\dots[\overline u,\underbrace{\overline v],\dots,\overline
    v}_n]=\one$?
\item[Ad-nilpotence problem] Given $v\in Q^*$ and assuming $\overline v$ is
  invertible, is it \emph{ad-nilpotent}, i.e.\ is $(g,\overline v)$ an
  Engel pair for all invertible $g\in G$?
\item[Orbit problem (OP)] Assume that a countable set $\Omega$ is
  given via a computable bijection with (say) $\N$, and that $G$ acts
  on the right on $\Omega$.  Given $\omega_1,\omega_2\in\Omega$, does
  there exist $g\in G$ with $\omega_1^g=\omega_2$? If so, which one?
\end{description}

In all cases, what is required is an algorithm that answers the
question. Equivalently, the inputs (a word, a finite list of words,
\dots) may be encoded into $\N$ by a computable bijection. Let
$\mathscr P\subseteq\N$ denote the set of inputs for which the answer
to the question is ``yes'. One then asks whether $\mathscr P$ is
recursive, namely whether there exists an algorithmic enumeration of
$\mathscr P$ and of $\N\setminus\mathscr P$.

There are $2^{\aleph_0}$ finitely generated semigroups up to
isomorphism, and only $\aleph_0$ algorithms, so for ``most'' finitely
semigroups all the above decision problems have a negative solution.

Note also that each of these decision problems are stated for a fixed
semigroup $G$. One may also ask directly some questions on $G$:
\begin{description}
\item[Semigroup structure] Is $G$ trivial? finite? commutative? free?
\item[Invertibility] Is the subgroup $G^\times$ of invertibles finite?
  If so, what is it?
\item[Group structure] Assume $G$ is a group. Is $G$ nilpotent? free?
  Engel, i.e.\ is every invertible element ad-nilpotent?
\end{description}
These questions make a lot of sense for a human, especially if the
semigroup is given implicitly as in Example~\ref{ex:grig} or
Example~\ref{ex:2}. They make no sense as decision problems: either
they hold or they don't, but they have an unequivocal answer for every
given $G$.

These last questions become much more interesting if one is given,
rather than a semigroup $G$, a countably infinite family of
semigroups. They could be given by semigroup presentations
\[G=\langle s_1,\dots,s_n\mid r_1=r'_1,\dots,r_m=r'_m\rangle;\]
or by ``self-similar presentations''
\begin{equation}\label{eq:pres}
  G=G(\Phi)
\end{equation}
meaning, in the sense of the previous section, the semigroup acting
faithfully on $X^*$ with the action given via~\eqref{eq:action} by
$\Phi\colon X\times Q\to Q^*\times X$.

%%%%%%%%%%%%%%%%%%%%%%%%%%%%%%%%%%%%%%%%%%%%%%%%%%%%%%%%%%%%%%%%
\section{Some negative results}
One seldom solves these decision problems directly. Rather, one uses
\emph{Turing reduction}: a problem $\mathscr A$ \emph{Turing-reduces}
to a problem $\mathscr B$ if there exists an algorithm answering
$\mathscr A$ given an oracle for $\mathscr B$. In this manner, if
$\mathscr A$ is unsolvable then so is $\mathscr B$.

It is a well-known fact that there are finitely-presented
semigroups~\cites{markoff:wpsemigroup,post:wpsemigroup}, and even
finitely-presented groups~\cite{novikov:wp}, with unsolvable word
problem. Indeed Turing machines may be encoded in semigroup
presentations, in such a manner that a word is trivial if and only if
the corresponding Turing machine computation halts.

Let $Q=\langle x_1,\dots,x_n\mid r_1,\dots,r_m\rangle$ be a finitely
presented group with unsolvable word problem. Mihailova considers
in~\cite{mihailova:directproducts}
\[M=\langle(x_1,x_1),\dots,(x_n,x_n),(\one,r_1),\dots,(\one,r_m)\rangle\subset
  F_n\times F_n.
\]
Then $(\one,u)\in M$ holds if and only if $\overline u=\one$ holds in
$Q$; so the membership problem is unsolvable in $F_2\times F_2$.
(Note however that the membership problem is solvable in free groups).

Now it is well-known that $F_n$ embeds in $\operatorname{GL}_2(\Z)$,
so $F_n\times F_n$ embeds in $\operatorname{GL}_4(\Z)$. Clearly, if
$G$ acts on the right on $\Omega$, and $H\le G$ is a subgroup with
$G,H$ finitely generated, and $\omega\in\Omega$ has trivial stabilizer
in $G$, then the membership problem for $H$ in $G$ Turing-reduces to
the orbit problem for $G$ acting on $\Omega$: given
$H=\langle v_1,\dots,v_n\rangle$ and $u\in G$, one has $u\in H$ if and
only if $\omega,\omega^u$ are in the same $H$-orbit.

In particular, if $K$ is a finitely generated group and $H$ acts on
$K$, then the orbit problem of $H$ on $K$ Turing-reduces to the
conjugacy problem in $K\rtimes H$. It follows that there exist
finitely generated subgroups of $\operatorname{GL}_4(\Z)$ with
unsolvable membership and conjugacy problems.

%%%%%%%%%%%%%%%%%%%%%%%%%%%%%%%%%%%%%%%%%%%%%%%%%%%%%%%%%%%%%%%%
\section{Automaton semigroups}
Consider a semigroup $G$ given by a self-similar presentation: there
are finite sets $X,Q$ and a map $\Phi\colon X\times Q\to Q^*\times X$,
defining a faithful action of $G=G(\Phi)$ on $X^*$
by~\eqref{eq:action}.

\begin{question}
  Is the word problem in $G$ decidable?
\end{question}

I suspect the answer is ``no'', but I don't know. Let us put
restrictions on $\Phi$ to make the problem more tractable.
\begin{definition}
  A \emph{Mealy automaton} is a map $\Phi\colon X\times Q\to Q\times X$.
\end{definition}
We display the automaton as a graph with stateset $Q$ and, for every
$x\in X,q\in Q$ with $\Phi(x,q)=(t,y)$, an edge starting in $q$,
ending in $t$, labeled `$(x,y)$' and called a \emph{transition}. The
letters $x$ and $y$ are respectively called the \emph{input} and
\emph{output} labels. Since the two formalisms are obviously
equivalent, we call \emph{automaton} either the map $\Phi$ or its
representation as a graph. The examples~\ref{ex:grig} and~\ref{ex:2}
above are Mealy automata, depicted respectively left and right in
Figure~\ref{fig:automata}.

\begin{figure}
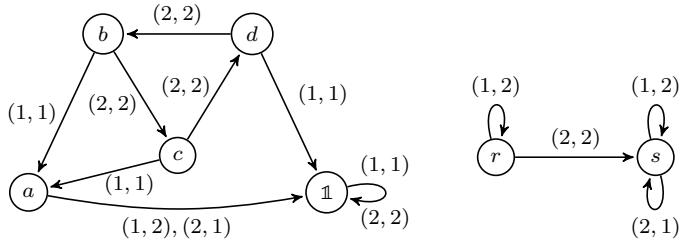

\[\begin{fsa}[scale=0.7]
  \node[state] (b) at (1.4,3) {$b$};
  \node[state] (d) at (4.2,3) {$d$};
  \node[state] (c) at (2.8,0.7) {$c$};
  \node[state] (a) at (0,0) {$a$};
  \node[state] (e) at (5.6,0) {$\one$};
  \path (b) edge node[left,pos=0.6] {$(2,2)$} (c) edge node[left] {$(1,1)$} (a)
        (c) edge node[left,pos=0.6] {$(2,2)$} (d) edge node[below right=-1mm] {$(1,1)$} (a)
        (d) edge node[above] {$(2,2)$} (b) edge node {$(1,1)$} (e)
        (a) edge[bend right=10] node[below] {$(1,2),(2,1)$} (e)
        (e) edge[loop right] node[above=1mm] {$(1,1)$} node[below=1mm] {$(2,2)$} (e);
\end{fsa}\qquad
\begin{fsa}[scale=0.7]
  \node[state] (r) at (0,3) {$r$};
  \node[state] (s) at (3,3) {$s$};
  \path (r) edge[loop above] node {$(1,2)$} (r) edge node[above] {$(2,2)$} (s)
        (s) edge[loop above] node[above] {$(1,2)$} (s) edge[loop below] node[below] {$(2,1)$} (s);
\end{fsa}\]
\caption{The automata generating the Grigorchuk group (left,
  Example~\ref{ex:grig}) and the Sushchanskyy semigroup (right,
  Example~\ref{ex:2})}\label{fig:automata}
\end{figure}

In this graph interpretation, the action of $\langle Q\rangle$ on
$X^*$ and on $X^\infty$ are directly visible: given $q\in Q$ and a
word $x_1x_2\dots$, find a path in the graph starting at $q$ and
having input labels $x_1,x_2,\dots$. Let the output label on this path
be $y_1,y_2,\dots$; then the result of the action is
\[(x_1x_2\dots)^q=y_1y_2\dots.\]

We call \emph{automaton semigroup} a self-similar semigroup presented
by a Mealy automaton as in~\eqref{eq:pres}, and we write the semigroup
$G=G(\Phi)$.

Let $\Phi$ be an automaton with stateset $Q$ and alphabet $X$. For
$x\in X$ and $q\in Q$, we denote by $q@x$ the endpoint of the
transition in $\Phi$ starting at $q$ and with input label $x$, and we
denote by $\pi(q)$ the transformation of $X$ induced by the edges
starting at $q$.  Thus every transition in the Mealy automaton gives
rise to
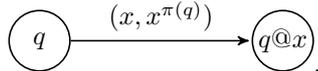
\[\begin{fsa}[scale=0.8]
  \node[state,minimum size=8mm] (q) at (0,0) {$q$};
  \node[state,inner sep=0pt,minimum size=8mm] (r) at (4,0) {$q@x$};
  \path (q) edge node[above] {$(x,x^{\pi(q)})$} (r);
\end{fsa}.
\]

\begin{proposition}
  The word problem in an automaton group is solvable in linear space
  (and therefore in exponential time).
\end{proposition}
\begin{proof}
  Let $\Phi$ be a Mealy automaton, and let $u,v\in Q^*$ be given
  words. By adding an identity state to $Q$, we may suppose
  $|u|=|v|=n$. Consider the graph with vertex set $Q^n$ and with an
  edge from $q_1\dots q_n$ to $t_1\dots t_n$ labeled `$(x_0,x_n)$'
  whenever there are edges from $q_i$ to $t_i$ labeled
  `$(x_{i-1},x_i)$' in $\Phi$ for all $i=1,\dots,n$. Then
  $\overline u=\overline v$ if and only if the following holds in this
  graph: the vertices $u,v$ may be identified, and outgoing edges with
  matching input may be identified, repeatedly, never causing an
  identification of vertices $y,z$ with different $\pi(y)\neq\pi(z)$.

  Since the graph is finite and every identification reduces its size,
  this proves that the word problem is decidable. By carefully
  arranging the order in which the graph is explored, this may be done
  in $\mathcal O(n)$ space.
\end{proof}

\begin{question}
  Is there an automaton for which the lower bound on the solution of
  the word problem is linear in space?
\end{question}

For every $n\in\N$, one may embed the matrix semigroup $M_n(\Z)$ into
an automaton semigroup. More precisely, consider the semigroup $G$ of
affine transformations $v\mapsto Av+w$, for all $A\in M_n(\Z)$ and
$w\in\Z^n$. Consider the alphabet $X=(\Z_{/2})^n$, and identify
$X^\infty$ with $n$-tuples of $2$-adics $(\Z_2)^n$. Let $G$ act on
$(\Z_2)^n$ by extending the natural action on $\Z$ by continuity. It
is easy to see that this makes $G$ a self-similar semigroup, and
furthermore every element of $G$ is contained in an automaton
subsemigroup of $G$.

This means that automaton semigroups are at least as powerful as
linear semigroups, and also shows that there exist automaton
semigroups with unsolvable conjugacy problem.

\begin{question}
  Do there exist automaton groups with unsolvable order problem?
\end{question}
It is known~\cite{gillibert:finiteness} that there exist automaton
semigroups with unsolvable order problem. This is proven by
Turing-reducing the order problem to a tiling problem.

\begin{definition}
  Let $\Phi\colon X\times Q\to Q\times X$ be a Mealy automaton. It is
  called \emph{bounded} if there is a constant $C$ such that, for all
  $n\in\N$, there are at most $C$ elements in
  $\Phi(X^n\times Q)\setminus(\{\one\}\times X^n)$.
\end{definition}
In terms of graphs describing the automaton, that condition says that,
apart from the identity state and its self-loops, there are no paths
in the automaton that follow more than one loop.

Bondarenko, Sidki and Zapata prove
in~\cite{bondarenko-b-s-z:conjugacy} that, if $G$ is an automaton
group generated by a bounded automaton, then the order problem is
solvable in $G$.

\begin{definition}
  Let $\Phi\colon X\times Q\to Q\times X$ be a Mealy automaton. It is
  called \emph{nuclear} if, for every $g\in Q^*$, there exists
  $n\in\N$ such that for all $u\in X^n$ we have
  $\overline{g@u}\in\overline Q$ as elements of $G(\Phi)$.
\end{definition}
In other words, for every $g\in G(\Phi)$, its action on all
remote-enough subtrees $uX^*$ may be described by elements of $Q$. An
automaton semigroup is called \emph{contracting} if it may be
presented by a nuclear automaton.

For instance, the Grigorchuk group $G_0$ is contracting, and the
automaton $\Phi$ presenting it is nuclear. On the other hand, the
semigroup $G_1$ is not contracting: any nuclear automaton presenting
it must contain the infinitely many distinct states $(rs)^n$ for all
$n\in\N$.

If an automaton $\Phi$ is nuclear, then this may be verified in finite
time: it suffices to check the condition of the definition for every
$g\in Q^2$.
\begin{question}
  Let $\Phi$ be an automaton. Is it decidable if $\Phi$ is nuclear? Or
  if the automaton $\Phi_{1,n}$ (on the stateset $Q^n$) is nuclear?
\end{question}

\begin{proposition}
  Let $\Phi$ be a nuclear automaton. Then the word problem in
  $G(\Phi)$ is solvable in polynomial time.
\end{proposition}
\begin{proof}
  Denote by $|\cdot|$ the word metric on $G(\Phi)$.  It follows from
  the definition that there is a constant $n$ such that
  $|\overline{g@v}|\le\frac12(|g|+1)$ for all $g\in G(\Phi),v\in
  X^n$. The preprocessing step is to compute which elements of $Q$ are
  equal in $G(\Phi)$, and to determine for each word $q_1q_2$ of
  length $2$ and each $v\in X^n$ an element $q\in Q$ such that
  $\overline{q_1q_2@v}=\overline q$.

  Then, given $g,h\in Q^*$, one computes $\#X^n$ words of length
  $\frac12(|g|+1),\frac12(|h|+1)$ representing the action on subtrees
  $u X^*$ for all $u\in X^n$; and compares them recursively. The
  complexity is polynomial of degree $n\log_2(\#X)$.
\end{proof}

\begin{question}
  Let $G(\Phi)$ be a contracting automaton semigroup. Is its torsion
  problem decidable?

  Let $G(\Phi)$ be a contracting automaton group. Is its conjugacy
  problem decidable?
\end{question}

We remarked in the introduction that the Grigorchuk group is an
infinite torsion group.
\begin{question}
  Is there an algorithm that, given an nuclear automaton $\Phi$,
  decides whether $G(\Phi)$ is infinite? Whether it is torsion?
\end{question}

\subsection{More constructions}
Let $G$ be a self-similar semigroup. We generalize the notation $q@x$
to arbitrary semigroup elements and words: consider a word $v\in X^*$
and an element $g\in G$; denote by $v^g$ the image of $v$ under
$g$. There is then a unique element of $G$, written $g@v$, with the
property
\[(v\,w)^g=(v^g)\,(w)^{g@v}\text{ for all }w\in X^*.
\]
We call by extension this element $g@v$ the \emph{state} of $g$ at
$v$; it is the state, in the Mealy automaton defining $g$, that is
reached from $g$ by following the path $v$ as input; thus in the
Grigorchuk automaton $b@1=a$ and $b@222=b$ and $(bc)@2=cd$. There is a
reverse construction: by $v*g$ we denote the transformation of $X^*$
(which need not belong to $G$) defined by
\[(v\,w)^{v*g}=v\,w^g,\qquad w^{v*g}=w\text{ if $w$ does not start with $v$}.
\]
Given a word $w=w_1\dots w_n\in X^*$ and a Mealy automaton $\Phi$ of
which $g$ is a state, it is easy to construct a Mealy automaton of which
$w*g$ is a state: add a path of length $n$ to $\Phi$, with input and
output $(w_1,w_1),\dots,(w_n,w_n)$ along the path, and ending at
$g$. Complete the automaton with transitions to the identity element. Then
the first vertex of the path defines the transformation $w*g$. For example,
here is $12*d$ in the Grigorchuk automaton:
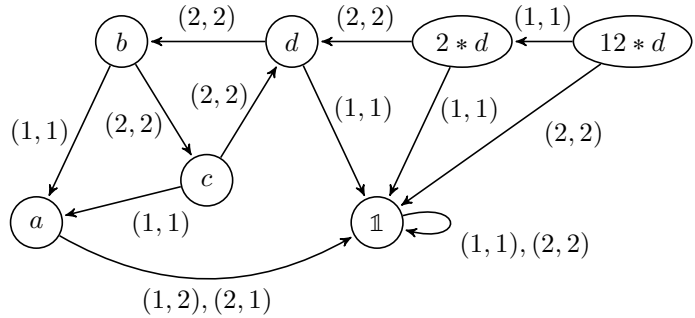
\[\begin{fsa}[scale=0.8,every state/.style={minimum size=4.5ex}]
  \node[state] (b) at (1.4,3) {$b$};
  \node[state] (d) at (4.2,3) {$d$};
  \node[state,ellipse] (2d) at (7.0,3) {$2*d$};
  \node[state,ellipse] (12d) at (9.8,3) {$12*d$};
  \node[state] (c) at (2.8,0.7) {$c$};
  \node[state] (a) at (0,0) {$a$};
  \node[state] (e) at (5.6,0) {$\one$};
  \path (b) edge node[left,pos=0.66] {$(2,2)$} (c) edge node[left] {$(1,1)$} (a)
        (c) edge node[left,pos=0.66] {$(2,2)$} (d) edge node {$(1,1)$} (a)
        (d) edge node[above] {$(2,2)$} (b) edge node[right,pos=0.33] {$(1,1)$} (e)
        (a) edge[bend right=30] node[below] {$(1,2),(2,1)$} (e)
        (e) edge[loop right] node[below right] {$(1,1),(2,2)$} (e)
        (12d) edge node[above] {$(1,1)$} (2d) edge node [pos=0.33] {$(2,2)$} (e)
        (2d) edge node[above] {$(2,2)$} (d) edge node [right,pos=0.33] {$(1,1)$} (e);
      \end{fsa}
\]
Note the simple identities $(g@v_1)@v_2=g@(v_1v_2)$,
$(v_1v_2)*g=v_1*(v_2*g)$, and $(v*g)@v=g$. Recall that we write
conjugation in $G$ as $g^h=h^{-1}g h$. For any $h\in G$ we have
\begin{equation}\label{eq:twist}
  (v*g)^h=v^h*(g^{h@v}).
\end{equation}

An automaton semigroup is called \emph{regular weakly branched} if
there exists a non-trivial subsemigroup $K$ of $G$ such that for every
$v\in X^*$ the semigroup $v*K$ is contained in $K$, and therefore also
in $G$. Ab\'ert proved in~\cite{abert:nonfree} that regular weakly
branched groups satisfy no law.

\subsection{Grigorchuk's example}\label{ss:grigorchuk}
The first Grigorchuk group $G_0$, defined in Example~\ref{ex:grig}, is
an automaton group which appears prominently in group theory, for
example as a finitely generated infinite torsion
group~\cite{grigorchuk:burnside} and as a group of intermediate word
growth~\cite{grigorchuk:growth}. This section is not an introduction
to Grigorchuk's first group, but rather a brief description of it with
all information vital for the calculation in~\S\ref{ss:proof}. For
more details, see e.g.~\cite{bartholdi-g-s:bg}.

Fix the alphabet $X=\{1,2\}$. The first Grigorchuk group $G_0$ is a
permutation group of the set of words $X^*$, generated by the four
non-trivial states $a,b,c,d$ of the automaton given in
Example~\ref{ex:grig}. Alternatively, the transformations $a,b,c,d$
may be defined recursively as follows:
\begin{equation}\label{eq:grigorchuk}
\begin{alignedat}{2}
  (1x_2\dots x_n)^a&=2x_2\dots x_n, &\qquad (2x_2\dots x_n)^a&=1x_2\dots x_n,\\
  (1x_2\dots x_n)^b&=1(x_2\dots x_n)^a, &\qquad (2x_2\dots x_n)^b&=2(x_2\dots x_n)^c,\\
  (1x_2\dots x_n)^c&=1(x_2\dots x_n)^a, &\qquad (2x_2\dots x_n)^c&=2(x_2\dots x_n)^d,\\
  (1x_2\dots x_n)^d&=1x_2\dots x_n, &\qquad (2x_2\dots x_n)^d&=2(x_2\dots x_n)^b
\end{alignedat}
\end{equation}
which directly follow from $d@1=\one$, $d@2=b$, etc.

It is remarkable that most properties of $G_0$ derive from a careful
study of the automaton (or equivalently this action), usually using
inductive arguments. For example,
\begin{proposition}[\cite{grigorchuk:burnside}]\label{prop:torsion}
  The group $G_0$ is infinite, and all its elements have order a power of $2$.
\end{proposition}

\noindent The self-similar nature of $G_0$ is made apparent in the
following manner:
\begin{proposition}[\cite{bartholdi-g:parabolic}*{\S4}]\label{prop:branch}
  Define $x=[a,b]$ and $K=\langle x,x^{c},x^{ca}\rangle$. Then $K$ is
  a normal subgroup of $G_0$ of index $16$, and $\psi(K)$ contains
  $K\times K$.

  In other words, for every $g\in K$ and every $v\in X^*$ the element
  $v*g$ belongs to $G_0$.
\end{proposition}

%%%%%%%%%%%%%%%%%%%%%%%%%%%%%%%%%%%%%%%%%%%%%%%%%%%%%%%%%%%%%%%%
\section{Engel Identities}
In this section, we restrict ourselves to invertible Mealy automata
and self-similar groups.

A \emph{law} in a group $G$ is a word $w=w(x_1,x_2,\dots,x_n)$ such
that $w(g_1,\dots,g_n)=\one$, the identity element, for all
$g_1,\dots,g_n\in G$; for example, commutative groups satisfy the law
$[x_1,x_2]=x_1^{-1}x_2^{-1}x_1x_2$. A \emph{variety} of groups is a
maximal class of groups satisfying a given law; e.g.\ the variety of
commutative groups (satisfying $[x_1,x_2]$) or of groups of exponent
$p$ (satisfying $x_1^p$);
see~\cites{neumann:varieties,straubing-weil:varieties}.

Consider now a sequence $\mathscr W=(w_0,w_1,\dots)$ of words in $n$
letters. Say that $(g_1,\dots,g_n)$ \emph{almost satisfies}
$\mathscr W$ if $w_i(g_1,\dots,g_n)=1$ for all $i$ large enough, and
say that $G$ \emph{almost satisfies $\mathscr W$} if all $n$-tuples
from $G$ almost satisfy $\mathscr W$. For example, $G$ almost
satisfies $(x_1,\dots,x_1^{i!},\dots)$ if and only if $G$ is a torsion
group.

The problem of deciding algorithmically whether a group belongs to a
given variety has received much attention~(see
e.g.~\cite{jackson:undecidable} and references therein); we consider
here the harder problems of determining whether a group (respectively
a tuple) almost satisfies a given sequence. This has, up to now, been
investigated mainly for the torsion sequence
above~\cite{godin-klimann-picantin:torsion-free}.

The Engel law is
\[E_c=E_c(x,y)=[x,y,\dots,y]=[\cdots[[x,y],y],\dots,y]\]
with $c$ copies of `$y$'; so $E_0(x,y)=x$, $E_1(x,y)=[x,y]$ and
$E_c(x,y)=[E_{c-1}(x,y),y]$. See below for a motivation. Let us call a
group (respectively a pair of elements) \emph{Engel} if it almost
satisfies $\mathscr E=(E_0,E_1,\dots)$.  Furthermore, let us call
$h\in G$ an \emph{Engel} element if $(g,h)$ is Engel for all $g\in G$.

\noindent A concrete consequence of our investigations is:
\begin{theorem}\label{thm:1}
  The first Grigorchuk group $G_0$ is not Engel. Furthermore, an
  element $h\in G_0$ is Engel if and only if $h^2=\one$.
\end{theorem}
We prove a similar statement for another prominent example of automaton
group, the \emph{Gupta-Sidki group}, see Theorem~\ref{thm:gs}.

Theorem~\ref{thm:1} follows from a partial algorithm, giving a
criterion for an element $y$ to be Engel. This algorithm proves, in
fact, that the element $ad$ in the Grigorchuk group is not Engel. Our
aim is to solve the following \textbf{decision problems} in an
automaton group $G$:
\begin{description}
\item[Engel($g,h$)] Given $g,h\in G$, does there exist $c\in\N$ with
  $E_c(g,h)=\one$?
\item[Engel($h$)] Given $h\in G$, does Engel($g,h$) hold for all $g\in G$?
\end{description}

\noindent The algorithm is described in~\S\ref{ss:algo}. As a consequence,
\begin{corollary}
  Let $G$ be an automaton group acting on the set of binary sequences
  $\{1,2\}^*$, that is contracting with contraction coefficient
  $\eta<1$. Then, for torsion elements $h$ of order $2^e$ with
  $2^{2^e}\eta<1$, the property Engel($h$) is decidable.
\end{corollary}

The Engel property attracted attention for its relation to nilpotency:
indeed a nilpotent group of class $c$ satisfies $E_c$, and conversely
among compact~\cite{medvedev:engel} and
solvable~\cite{gruenberg:engel} groups, if a group satisfies $E_c$ for
some $c$ then it is locally nilpotent. Conjecturally, there are
non-locally nilpotent groups satisfying $E_c$ for some $c$, but this
is still unknown. It is also an example of iterated identity,
see~\cites{erschler:iteratedidentities,bandman+:dynamics}. In
particular, the main result of~\cite{bandman+:dynamics} implies easily
that the Engel property is decidable in algebraic groups.

It is comparatively easy to prove that the first Grigorchuk group $G_0$
satisfies no law~\cites{abert:nonfree,leonov:identities}; this result
holds for a large class of automaton groups. In fact, if a group
satisfies a law, then so does its profinite completion. In the class
mentioned above, the profinite completion contains abstract free
subgroups, precluding the existence of a law. No such arguments would
help for the Engel property: the restricted product of all finite
nilpotent groups is Engel, but the unrestricted product again contains
free subgroups. This is one of the difficulties in dealing with
iterated identities rather than identities.

If $\mathfrak A$ is a nil algebra (namely, for every $a\in\mathfrak A$
there exists $n\in\N$ with $a^n=0$) then the set of elements of the
form $\{\one+a:a\in\mathfrak A\}$ forms a group $1+\mathfrak A$ under the
law $(\one+a)(\one+b)=\one+(a+b+ab)$. If $\mathfrak A$ is defined over a field
of characteristic $p$, then $\one+\mathfrak A$ is a torsion group since
$(\one+a)^{p^n}=\one$ if $a^{p^n}=0$. Golod constructed
in~\cite{golod:burnside} non-nilpotent nil algebras $\mathfrak A$ all
of whose $2$-generated subalgebras are nilpotent (namely,
$\mathfrak A^n=0$ for some $n\in\N$); given such an $\mathfrak A$, the
group $\one+\mathfrak A$ is Engel but not locally nilpotent.

Golod introduced these algebras as means of obtaining infinite,
finitely generated, residually finite (every non-trivial element in
the group has a non-trivial image in some finite quotient), torsion
groups. Golod's construction is highly non-explicit, in contrast with
Grigorchuk's group for which much can be derived from the automaton's
properties.

It is therefore of high interest to find explicit examples of Engel
groups that are not locally nilpotent, and the methods and algorithms
presented here are a step in this direction.

In the remainder of this text, we concentrate on the Engel property,
which is equivalent to nilpotency for finite groups. In particular, if
an automaton group $G$ is to have a chance of being Engel, then its
image under the map $\pi\colon G\to\sym(X)$ should be a nilpotent
subgroup of $\sym(X)$. Since finite nilpotent groups are direct
products of their $p$-Sylow subgroups, we may reduce to the case in
which the image of $G$ in $\sym(X)$ is a $p$-group. A further
reduction lets us assume that the image of $G$ is an abelian subgroup
of $\sym(X)$ of prime order. We therefore make the following
\begin{assumption}
  The alphabet is $X=\{1,\dots,p\}$ and automaton groups $G(\Phi)$ are
  generated by automata $\Phi\colon X\times Q\to Q\times X$ such that
  for every $q\in Q$ the corresponding map $\pi(q)\in\sym(X)$
  describing the action of $q$ on $X$ takes values in the cyclic
  subgroup $\Z_{/p}$ of $\sym(X)$ generated by the cycle
  $(1,2,\dots,p)$.
\end{assumption}

We make a further reduction in that we only consider the Engel
property for elements of finite order. This is not a very strong
restriction: given $h$ of infinite order, one can usually find an
element $g\in G$ such that the conjugates $\{g^{h^n}:n\in\Z\}$ are
independent, and it then follows that $h$ is not Engel. We content
ourselves with an example:
\begin{example}[The Brunner-Sidki-Vieira group~\cite{brunner-s-v:nonsolvable}]
  The generating set is $Q=\{\tau^{\pm1},\mu^{\pm1},\one\}$ and the
  alphabet is $X=\{1,2\}$. The self-similarity structure is given by
  \begin{xalignat*}{5}
    \autsq\Phi \tau12\one && \autsq\Phi \tau21\tau && \autsq\Phi \mu12\one && \autsq\Phi \mu21{\mu^{-1}}.
  \end{xalignat*}
  Let $G_2$ denote the group generated by $Q$. The elements $\tau,\mu$
  have infinite order, and in fact act transitively on $X^n$ for all
  $n$.

  Let us show that $(\mu,\tau)$ is not an Engel pair, namely
  $E_c(\mu,\tau)\neq\one$ for all $c\in\N$. We rely on the
  calculations in~\cite{bartholdi:bsvlcs}, which compute the lower
  $2$-central series of $G_2$, namely the series of subgroups
  $\delta_1=G_2$ and
  $\delta_{n+1}=[\delta_n,G_2]\{g^2:g\in \delta_n\}$ for for all
  $n\ge1$. In that article, a basis of the $\mathbb F_2$-vector space
  $\delta_n/\delta_{n+1}$ is given for all $n$, and in particular one
  of the basis vectors of $\delta_n/\delta_{n+1}$ is
  $E_{n-1}(\mu,\tau)$.
\end{example}

%%%%%%%%%%%%%%%%%%%%%%%%%%%%%%%%%%%%%%%%%%%%%%%%%%%%%%%%%%%%%%%%
\section{A semi-algorithm for deciding the Engel property}\label{ss:algo}

We start by describing a semi-algorithm to check the Engel property. It
will sometimes not return any answer, but when it returns an answer then that
answer is guaranteed correct. It is guaranteed to terminate as long as the
contraction property of the automaton group $G$ is strong enough.

\begin{algorithm}\label{algo:1}
  Let $G$ be a contracting automaton group with alphabet
  $X=\{1,\dots,p\}$ for prime $p$, with the contraction property
  $\|g@j\|\le\eta\|g\|+C$.

  For $n\in p\N$ and $R\in\R$ consider the following finite graph
  $\Gamma_{n,R}$. Its vertex set is $B(R)^n\cup\{\texttt{fail}\}$,
  where $B(R)$ denotes the set of elements of $G$ of length at most
  $R$. Its edge set is defined as follows: consider a vertex
  $(g_1,\dots,g_n)$ in $\Gamma_{n,R}$, and compute
  \[(h_1,\dots,h_n)=(g_1^{-1}g_2,\dots,g_n^{-1}g_1).\]
  If $h_i$ fixes $X$ for all $i$, i.e.\ all $h_i$ have trivial image
  in $\sym(X)$, then for all $j\in\{1,\dots,p\}$ there is an edge from
  $(g_1,\dots,g_n)$ to $(h_1@j,\dots,h_n@j)$, or to \texttt{fail} if
  $(h_1@j,\dots,h_n@j)\not\in B(R)^n$. If some $h_i$ does not fix $X$,
  then there is an edge from $(g_1,\dots,g_n)$ to $(h_1,\dots,h_n)$,
  or to \texttt{fail} if $(h_1,\dots,h_n)\not\in B(R)^n$.
  \begin{description}
  \item[\boldmath Given $g,h\in G$ with $h^n=\one$:] Set
    $t_0=(g,g^h,g^{h^2},\dots,g^{h^{n-1}})$. If there exists $R\in\N$
    such that no path in $\Gamma_{n,R}$ starting at $t_0$ reaches
    \texttt{fail}, then Engel($g,h$) holds if and only if the only
    cycle in $\Gamma_{n,R}$ reachable from $t_0$ passes through
    $(\one,\dots,\one)$.

    If the contraction coefficient satisfies $2^n\eta<1$, then it is
    sufficient to consider $R=(\|g\|+n\|h\|)2^n C/(1-2^n\eta)$.
  \item[\boldmath Given $n\in\N$:] The Engel property holds for all
    elements of exponent $n$ if and only if, for all $R\in\N$, the
    only cycle in $\Gamma_{n,R}$ passes through $(\one,\dots,\one)$.

    If the contraction coefficient satisfies $2^n\eta<1$, then it is
    sufficient to consider $R=2^n C/(1-2^n\eta)$.
  \item[\boldmath Given $G$ weakly branched and $n\in\N$:] If for some
    $R\in\N$ there exists a cycle in $\Gamma_{n,R}$ that passes
    through an element of $K^n\setminus\one^n$, then no element of $G$
    whose order is a multiple of $n$ is Engel.

    If the contraction coefficient satisfies $2^n\eta<1$, then it is
    sufficient to consider $R=2^n C/(1-2^n\eta)$.
  \end{description}
\end{algorithm}

We consider the graphs $\Gamma_{n,R}$ as subgraphs of a graph
$\Gamma_{n,\infty}$ with vertex set $G^n$ and same edge definition as the
$\Gamma_{n,R}$.

We note first that, if $G$ satisfies the contraction condition
$2^n\eta<1$, then all cycles of $\Gamma_{n,\infty}$ lie in fact in
$\Gamma_{n,2^n C/(1-2^n\eta)}$. Indeed, consider a cycle passing
through $(g_1,\dots,g_n)$ with $\max_i\|g_i\|=R$. Then the cycle
continues with $(g^{(1)}_1,\dots,g^{(1)}_n)$,
$(g^{(2)}_1,\dots,g^{(2)}_n)$, etc.\ with $\|g_i^{(k)}\|\le2^k R$; and
then for some $k\le n$ we have that all $g_i^{(k)}$ fix $X$; namely,
they have a trivial image in $\sym(X)$, and the map $g\mapsto g@j$ is
an injective homomorphism on them. Indeed, let
$\pi_1,\dots,\pi_n,\pi^{(i)}_1,\dots,\pi^{(i)}_n\in\Z_{/p}\subset\sym(X)$
be the images of $g_1,\dots,g_n,g^{(i)}_1,\dots,g^{(i)}_n$
respectively, and denote by $S\colon\Z_{/p}^n\to\Z_{/p}^n$ the cyclic
permutation operator. Then
$(\pi^{(n)}_1,\dots,\pi^{(n)}_n)=(S-1)^n(\pi_1,\dots,\pi_n)$, and
$(S-\one)^n=\sum_j S^j\binom n j=0$ since $p|n$ and $S^n=\one$. Thus there is
an edge from $(g^{(k)}_1,\dots,g^{(k)}_n)$ to
$(g^{(k+1)}_1@j,\dots,g^{(k+1)}_n@j)$ with
$\|g^{(k+1)}_i@j\|\le\eta\|g^{(k)}_i\|+C\le\eta2^n R+C$. Therefore, if
$R>2^n C/(1-2^n\eta)$ then $2^n\eta R+C<R$, and no cycle can return to
$(g_1,\dots,g_n)$.

Consider now an element $h\in G$ with $h^n=\one$. For all $g\in G$, there
is an edge in $\Gamma_{n,\infty}$ from $(g,g^h,\dots,g^{h^{n-1}})$ to
$([g,h]@v,[g,h]^h@v,[g,h]^{h^{n-1}}@v)$ for some word
$v\in\{\varepsilon\}\sqcup X$, and therefore for all $c\in\N$ there
exists $d\le c$ such that, for all $v\in X^d$, there is a length-$c$
path from $(g,g^h,\dots,g^{h^{n-1}})$ to
$(E_c(g,h)@v,\dots,E_c(g,h)^{h^{n-1}}@v)$ in $\Gamma_{n,\infty}$.

We are ready to prove the first assertion: if Engel($g,h$), then
$E_c(g,h)=\one$ for some $c$ large enough, so all paths of length $c$ starting
at $(g,g^h,\dots,g^{h^{n-1}})$ end at $(\one,\dots,\one)$. On the other hand, if
Engel($g,h$) does not hold, then all long enough paths starting at
$(g,g^h,\dots,g^{h^{n-1}})$ end at vertices in the finite graph
$\Gamma_{n,2^n C/(1-2^n\eta)}$ so must eventually reach cycles; and one of
these cycles is not $\{(\one,\dots,\one)\}$ since $E_c(g,h)\neq\one$ for all $c$.

The second assertion immediately follows: if there exists $g\in G$
such that Engel($g,h$) does not hold, then again a non-trivial cycle
is reached starting from $(g,g^h,\dots,g^{h^{n-1}})$, and
independently of $g,h$ this cycle belongs to the graph
$\Gamma_{n,2^n C/(1-2^n\eta)}$.

For the third assertion, let
$\bar k=(k_1,\dots,k_n)\in K^n\setminus\one^n$ be a vertex of a cycle in
$\Gamma_{n,2^n C/(1-2^n\eta)}$. Consider an element $h\in G$ of order
$s n$ for some $s\in\N$. By the condition that $\#X=p$ is prime and the
image of $G$ in $\sym(X)$ is a cyclic group, $s n$ is a power of $p$,
so there exists an orbit $\{v_1,\dots,v_{s n}\}$ of $h$, so labeled
that $v_i^h=v_{i-1}$, indices being read modulo $s n$. For
$i=1,\dots,s n$ define
\[h_i=(h@v_1)^{-1}\cdots(h@v_i)^{-1},
\]
noting $h_i(h@v_i)=h_{i-1}$ for all $i=1,\dots,s n$ since
$h^{s n}=\one$. Denote by `$i\%n$' the unique element of $\{1,\dots,n\}$
congruent to $i$ modulo $n$, and consider the element
\[g=\prod_{i=1}^{s n}\big(v_i*k_{i\%n}^{h_i}\big),
\]
which belongs to $G$ since $G$ is weakly branched. Let
$(k_1^{(1)},\dots,k_n^{(1)})$ be the next vertex on the cycle of
$\bar k$. We then have, using~\eqref{eq:twist},
\[[g,h]=g^{-1}g^h=\prod_{i=1}^{s n}\big(v_i*k_{i\%n}^{-h_i}\big)\prod_{i=1}^{s n}\big(v_{i-1}*k_{i\%n}^{h_i(h@v_i)}\big)=\prod_{i=1}^{s n}\big(v_i*(k_{i\%n}^{(1)})^{h_i}\big),\]
and more generally $E_c(g,h)$ and some of its states are read off the
cycle of $\bar k$. Since this cycle goes through non-trivial group
elements, $E_c(g,h)$ has a non-trivial state for all $c$, so is
non-trivial for all $c$, and Engel($g,h$) does not hold.

%%%%%%%%%%%%%%%%%%%%%%%%%%%%%%%%%%%%%%%%%%%%%%%%%%%%%%%%%%%%%%%%
\section{Proof of Theorem~\ref{thm:1}}\label{ss:proof}
The Grigorchuk group $G_0$ is contracting, with contraction
coefficient $\eta=1/2$. Therefore, the conditions of validity of
Algorithm~\ref{algo:1} are not satisfied by the Grigorchuk group, so
that it is not guaranteed that the algorithm will succeed, on a given
element $h\in G_0$, to prove that $h$ is not Engel. However, nothing
forbids us from running the algorithm with the hope that it
nevertheless terminates. It seems experimentally that the algorithm
always succeeds on elements of order $4$, and the argument proving the
third claim of Algorithm~\ref{algo:1} (repeated here for convenience)
suffices to complete the proof of Theorem~\ref{thm:1}.

Below is a self-contained proof of Theorem~\ref{thm:1}, extracting the
relevant properties of the previous section, and describing the computer
calculations as they were keyed in.

Consider first $h\in G_0$ with $h^2=\one$. It follows from
Proposition~\ref{prop:torsion} that $h$ is Engel: given $g\in G_0$, we
have $E_{1+k}(g,h)=[g,h]^{(-2)^k}$ so $E_{1+k}(g,h)=\one$ for $k$ larger
than the order of $[g,h]$.

For the other case, we start by a side calculation. In the Grigorchuk
group $G_0$, define $x=[a,b]$ and $K=\langle x\rangle^{G_0}$ as in
Proposition~\ref{prop:branch}, consider the quadruple
\[A_0=(A_{0,1},A_{0,2},A_{0,3},A_{0,4})=(x^{-2}x^{2ca},\,x^{-2ca}x^2x^{2cab},\,x^{-2cab}x^{-2},\,x^2)\]
of elements of $K$, and for all $n\ge0$ define
\[A_{n+1}=(A_{n,1}^{-1}A_{n,2},\,A_{n,2}^{-1}A_{n,3},\,A_{n,3}^{-1}A_{n,4},\,A_{n,4}^{-1}A_{n,1}).\]
\begin{lemma}\label{lem:calculation}
  For all $i=1,\dots,4$, the element $A_{9,i}$ fixes $111112$, is
  non-trivial, and satisfies $A_{9,i}@111112=A_{0,i}$.
\end{lemma}
\begin{proof}
  This is proven purely by a computer calculation. It is performed as
  follows within \textsc{Gap}:
\begin{Verbatim}[commandchars=\\\{\}]
gap> \textit{LoadPackage("FR");;}
gap> \textit{AssignGeneratorVariables(GrigorchukGroup);;}
gap> \textit{x2 := Comm(a,b)^2;; x2ca := x2^(c*a);; one := a^0;;}
gap> \textit{A0 := [x2^-1*x2ca,x2ca^-1*x2*x2ca^b,(x2ca^-1)^b*x2^-1,x2];;}
gap> \textit{v := [1,1,1,1,1,2];; A := A0;; }
gap> \textit{for n in [1..9] do A := List([1..4],i->A[i]^-1*A[1+i mod 4]); od;}
gap> \textit{ForAll([1..4],i->v^A[i]=v and A[i]<>one and State(A[i],v)=A0[i]);}
true
\end{Verbatim}
\end{proof}

Consider now $h\in G_0$ with $h^2\neq\one$. Again by
Proposition~\ref{prop:torsion}, we have $h^{2^e}=\one$ for some minimal
$e\in\N$, which is furthermore at least $2$. We keep the notation
`$a\%b$' for the unique number in $\{1,\dots,b\}$ that is congruent to
$a$ modulo $b$.

Let $n$ be large enough so that the action of $h$ on $X^n$ has an
orbit $\{v_1,v_2,\dots,v_{2^e}\}$ of length $2^e$, numbered so that
$v_{i+1}^h=v_i$ for all $i$, indices being read modulo $2^e$. For
$i=1,\dots,2^e$ define
\[h_i=(h@v_1)^{-1}\cdots(h@v_i)^{-1},
\]
noting $h_i(h@v_i)=h_{i-1\% 2^e}$ for all $i=1,\dots,2^e$ since $h^{2^e}=\one$,
and consider the element
\[g=\prod_{i=1}^{2^e}\big(v_i*A_{0,i\%4}^{h_i}\big),
\]
which is well defined since $4|2^e$ and belongs to $G_0$ by
Proposition~\ref{prop:branch}. We then have, using~\eqref{eq:twist},
\[[g,h]=g^{-1}g^h=\prod_{i=1}^{2^e}\big(v_i*A_{0,i\%4}^{-h_i}\big)\prod_{i=1}^{2^e}\big(v_{i-1\% 2^e}*A_{0,i\%4}^{h_i(h@v_i)}\big)=\prod_{i=1}^{2^e}\big(v_i*A_{1,i}^{h_i}\big),\]
and more generally
\[E_c(g,h)=\prod_{i=1}^{2^e}\big(v_i*A_{c,i}^{h_i}\big).\]
Therefore, by Lemma~\ref{lem:calculation}, for every $k\ge0$ we have
$E_{9k}(g,h)@v_0(111112)^k=A_{0,1}\neq\one$, so $E_c(g,h)\neq\one$ for all
$c\in\N$ and we have proven that $h$ is not an Engel element.

%%%%%%%%%%%%%%%%%%%%%%%%%%%%%%%%%%%%%%%%%%%%%%%%%%%%%%%%%%%%%%%%
\section{Other examples}
Similar calculations apply to the Gupta-Sidki group $\Gamma$
introduced in~\cite{gupta-s:burnside}. This is another example of
infinite torsion group, acting on $X^*$ for $X=\{1,2,3\}$ and
generated by the states of the following automaton:
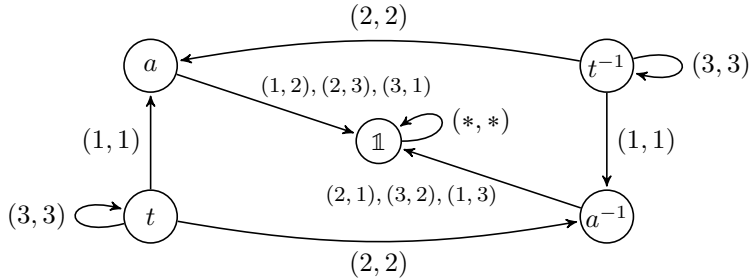
\[\begin{fsa}
  \node[state,inner sep=0pt,minimum size=7mm] (t) at (-3,-1) {$t$};
  \node[state,inner sep=0pt,minimum size=7mm] (T) at (3,1) {$t^{-1}$};
  \node[state,inner sep=0pt,minimum size=7mm] (a) at (-3,1) {$a$};
  \node[state,inner sep=0pt,minimum size=7mm] (A) at (3,-1) {$a^{-1}$};
  \node[state] (e) at (0,0) {$\one$};
  \path (t) edge node {$(1,1)$} (a)
  edge [bend right=10] node[below] {$(2,2)$} (A)
  edge [loop left] node {$(3,3)$} (t)
  (T) edge node {$(1,1)$} (A)
  edge [bend right=10] node[above] {$(2,2)$} (a)
  edge [loop right] node {$(3,3)$} (T)
  (a) edge node [pos=0.2,right=5mm] {\small $(1,2),(2,3),(3,1)$} (e)
  (A) edge node [pos=0.2,left=5mm] {\small $(2,1),(3,2),(1,3)$} (e)
  (e) edge [loop,in=35,out=0,min distance=8mm,looseness=10] node[right] {$(*,*)$} (e);
\end{fsa}\]
The transformations $a,t$ may also be defined recursively by
\begin{equation}\label{eq:guptasidki}
\begin{alignedat}{3}
  (1v)^a&=2v, &\qquad (2v)^a&=3v, &\qquad (3v)^a&=1v,\\
  (1v)^t&=1v^a, &\qquad (2v)^t&=2v^{a^{-1}}, &\qquad (3v)^t&=3v^t.
\end{alignedat}
\end{equation}
The Gupta-Sidki group is contracting, with contraction coefficient
$\eta=1/2$. Again, this is not sufficient to guarantee that
Algorithm~\ref{algo:1} terminates, but it nevertheless did succeed in
proving
\begin{theorem}\label{thm:gs}
  The only Engel element in the Gupta-Sidki group $\Gamma$ is the identity.
\end{theorem}
We only sketch the proof, since it follows that of Theorem~\ref{thm:1}
quite closely. Analogues of Propositions~\ref{prop:torsion}
and~\ref{prop:branch} hold, with $[\Gamma,\Gamma]$ in the r\^ole of
$K$. An analogue of Lemma~\ref{lem:calculation} holds with
$A_0=([a^{-1},t],[a,t]^a,[t^{-1},a^{-1}])$ and $A_{4,i}@122=A_{0,i}$.

%%%%%%%%%%%%%%%%%%%%%%%%%%%%%%%%%%%%%%%%%%%%%%%%%%%%%%%%%%%%%%%%
\section{Closing remarks}
An important feature of automaton groups is their amenability to
computer experiments, and even as in this case of rigorous
verification of mathematical assertions;
see also~\cite{klimann-mairesse-picantin:computations}, and the numerous
decidability and undecidability of the finiteness property
in~\cites{akhavi-klimann-lombardy-mairesse-picantin:finiteness,gillibert:finiteness,klimann:finiteness}.

The proof of Theorem~\ref{thm:1} relies on a computer calculation. It could
be checked by hand, at the cost of quite unrewarding effort. One of the
purposes of this article is, precisely, to promote the use of computers
in solving general questions in group theory: the calculations performed,
and the computer search involved, are easy from the point of view of a
computer but intractable from the point of view of a human.

The calculations were performed using the author's group theory
package \textsc{Fr}, specially written to manipulate automaton
groups. This package integrates with the computer algebra system
\textsc{Gap}~\cite{gap4:manual}, and is freely available from the
\textsc{Gap} distribution site
\[\verb+http://www.gap-system.org+\]

It would be dishonest to withhold from the reader how I arrived at the
examples given for the Grigorchuk and Gupta-Sidki groups. I started with
small words $g,h$ in the generators of $G_0$, respectively $\Gamma$,
and computed $E_c(g,h)$ for the first few values of $c$. These
elements are represented, internally to \textsc{Fr}, as Mealy
automata. A natural measure of the complexity of a group element is
the size of the minimized automaton, which serves as a canonical
representation of the element.

For some choices of $g,h$ the size increases exponentially with $c$,
limiting the practicality of computer experiments. For others (such as
$(g,h)=((b a)^4c,ad)$ for the Grigorchuk group), the size increases
roughly linearly with $c$, making calculations possible for $c$ in the
hundreds. Using these data, I guessed the period $p$ of the recursion
($9$ in the case of the Grigorchuk group), and searched among the
states of $E_c(g,h)$ and $E_{c+p}(g,h)$ for common elements; in the
example, I found such common states for $c=23$. I then took the
smallest-size quadruple of states that appeared both in $E_c(g,h)$ and
$E_{c+p}(g,h)$ and belonged to $K$, and expressed the calculation
taking $E_c(g,h)$ to $E_{c+p}(g,h)$ in the form of
Lemma~\ref{lem:calculation}.

It was already shown by Bludov~\cite{bludov:engel} that the wreath
product $G_0^4\rtimes D_4$ is not Engel. He gave, in this manner, an
example of a torsion group in which a product of Engel elements is not
Engel. Our proof is a refinement of his argument. In fact, his result
may also be used to obtain another proof of the fact that $G_0$ is not
Engel: the Grigorchuk contains a copy of $D_4$, say generated by
$a,d$, which has an orbit of size $4$, for example
$\{111,112,211,212\}$. The branching subgroup $K$ contains a subgroup,
for example the stabilizer $K_{111}$ of $111$, which maps onto $G_0$
by restriction to the subtree $111X^*$. The Grigorchuk group therefore
contains the subgroup
$\langle 111*K_{111},a,d\rangle\cong K_{111}^4\rtimes D_4$ which maps
onto the non-Engel group $G_0^4\rtimes D_4$, so $G_0$ itself is not
Engel.

A direct search for the elements $A_{0,1},\dots,A_{0,4}$ appearing in
the proof of Theorem~\ref{thm:1} would probably not be successful, and
has not yielded simpler elements than those given before
Lemma~\ref{lem:calculation}, if one restricts them to belong to $K$;
one can only wonder how Bludov found the quadruple $(\one,d,ca,ab)$,
presumably without the help of a computer.

%%%%%%%%%%%%%%%%%%%%%%%%%%%%%%%%%%%%%%%%%%%%%%%%%%%%%%%%%%%%%%%%
\section*{Acknowledgments}
I am deeply grateful to Anna Erschler for stimulating my interest in
this question and for having suggested a computer approach to the
problem, and to Ines Klimann and Matthieu Picantin for helpful
discussions that have improved the presentation of this text.

%%%%%%%%%%%%%%%%%%%%%%%%%%%%%%%%%%%%%%%%%%%%%%%%%%%%%%%%%%%%%%%%

\end{document}